\documentclass{article}
 \newcommand {\theoremstyle} [1] { }
 \newenvironment{proof}{{\noindent\it\underline{Proof}}:}{\hfill$\Box$}
  
 \newtheorem{thm}{Theorem}[section]
 \newtheorem{prop}{Proposition}[section]
 \theoremstyle{plain}
 \newtheorem{cor}{Corollary}[section]
 \theoremstyle{definition}

 \theoremstyle{remark}
 \newtheorem{rem}[thm]{Remark}

\usepackage {amssymb}
\usepackage{graphicx}
\usepackage{mathtools}
\usepackage{xfrac}
\usepackage{xcolor}

\def\R{\mathbb{R}}

\def \ee{\varepsilon}
\def \a{\alpha}
\def \b{\beta}
\def\I{\mathcal{I}}
\def\C{\mathcal{C}}
\def\Z{\mathbb{Z}}

\usepackage {amssymb}

\makeatother

\begin{document}

\author{Pablo Amster}

\title{On  a theorem by Browder and its application to nonlinear boundary value problems\footnote{The present manuscript corresponds to the  version initially accepted for publication at Bulletin of the London Mathematical Society (to appear).}}
\date{}
\maketitle 
\begin{center}
    {\bf\large}\vspace{1mm}
{\small Departamento de Matem\'atica, Facultad de Ciencias Exactas y Naturales.}\\ 
{\small Universidad de Buenos Aires \& IMAS-CONICET.}\\
{\small Ciudad Universitaria Pabell\'on I (1428), Buenos Aires, Argentina.}\\
{\small Email:   pamster@dm.uba.ar} 
\end{center}

\begin{abstract}
In a paper from 1960, Felix Browder 
established a theorem concerning 
the continuation of the fixed points of a 
family of continuous 
functions $f_t:X\to X$ depending continuously on a parameter $t\in [0,1]$, 
where $X$ is a convex and compact subset of $\R^n$. 
Here, the    result is 
presented for a compact mapping $f:A\times X\to X$ where $X$ is a 
convex, closed and bounded subset of  an arbitrary normed space and $A$ is an arcwise connected  topological space.  Applications 
to nonlinear boundary value problems are given; specifically, we shall  present new viewpoints of  known results, introduce some novel results and exhibit some open problems. 

\end{abstract}

\section{Introduction}

A theorem established by Browder in \cite{browder} guarantees that if  $f:[0,1]\times X\to X$ is continuous, with $X\subset \R^n$ is compact and   convex, 
then the set $$\mathcal F_{[0,1]}:= \{(t,x): f(t,x)=x\}= \bigcup_{t\in [0,1]} \mathcal \{t\}\times \textup{Fix}(f_t)
$$
has a connected component $\C$ that intersects 
both $\mathcal F_0:=\{0\}\times \textup{Fix}(f_0)$ and $\mathcal F_1:=\{1\}\times \textup{Fix}(f_1)$. 
We shall refer to this situation 
by saying that $\C$ connects $\mathcal F_0$ and $\mathcal F_1$. 
Here,  the notation
$\textup{Fix}(g)$ stands for the set of fixed points of a mapping $g$ and, as usual,  $f_t(x):=f(t,x)$. 
A trivial application could be the following version of the Poincar\'e-Miranda Theorem:
if  $\phi:[0,1]^2\to \R^2$ is  continuous and satisfies
$$\phi_1(0,x)< 0 <\phi_1(1,x)\qquad x\in [0,1]
$$
$$\phi_2(t,0)< 0 <\phi_2(t,1)\qquad t\in [0,1]
$$
then $\phi$ vanishes at some $(t,x)\in (0,1)^2$. 
Indeed, it suffices to fix $M>0$ large enough such that the function 
 $f(t,x):= x - \frac  {\phi_2(t,x)}M$ satisfies 
$f([0,1]^2)\subset [0,1]$. Taking    $\C\subset\mathcal F_{[0,1]}$ that connects $\mathcal F_0$ with $\mathcal F_1$ and noticing that $\phi|_{\mathcal F_0} < 0 < \phi|_{\mathcal F_1}$, 
it is seen that $\phi_1$ vanishes at some     $(t,x)\in \C$ and, consequently, $\phi(t,x)=(0,0)$.     This example may give the illusion of an inductive 
proof of the Brouwer theorem because, if Browder's result is assumed 
for $n$ and 
$f=(f_1,f_2)$ is   a continuous mapping from  
$[0,1]\times [0,1]^n$ into itself, then there exists a continuum $\mathcal C$ of points $(t,x)$ such that $f_2(t,x)=x$ and goes from $t=0$ to $t=1$. Thus, the existence of a fixed point of $f$ follows, because  the continuous 
mapping $\Phi(t,x):= t-f_1(t,x)$ changes sign over $\mathcal C$. However, 
as said in \cite{soso2}, the $n$-dimensional Browder theorem 
is already an equivalent form of the $n+1$-dimensional Brouwer theorem, so the inductive argument fails. 

Our applications to boundary value problems will be based on the fact that the result is still valid 
when $X$ is (homeomorphic to) a 
convex, closed and bounded subset of  an arbitrary normed space $E$, provided that $\textup{Im}(f)$ has compact closure. As a consequence, the following result is deduced in a straightforward manner:

\begin{thm}\label{main}
Let $E$ be a normed space and let 
$X\subset E$ be bounded, closed and convex. Consider a continuous mapping $f:A\times X\to X$, where $A$ is an arcwise connected topological space. Assume that  $\overline{f(K\times X)}$ is compact for any compact subset $K\subset A$. Then for each $\a,\b\ \in A$ there exists a connected set 
$$\C\subset \mathcal F_A:= \{(\sigma,x)\in A\times X:f(\sigma,x)=x\}=
\bigcup_{\sigma\in A} \{\sigma\}\times \textup{Fix}(f_\sigma):= \bigcup_{\sigma\in A}\mathcal F_\sigma$$ such that 
$$\C\cap \mathcal F_\a  \ne \emptyset, \qquad
\C\cap \mathcal F_\b\ne \emptyset.
$$
\end{thm}

Indeed, Theorem \ref{main} is deduced from the particular case $A=[0,1]$ as follows. For $\a, \b\in A$, take a continuous curve $\gamma:[0,1]\to A$ such that $\gamma(0)=\a$, $\gamma(1)=\b$ and define $f^\gamma:[0,1]\times X\to X$ by $f^\gamma(t,x):=f(\gamma(t),x)$. Because $f^\gamma$ is continuous and $K:=\gamma([0,1])$ is compact, the particular case implies the existence of $\C_\gamma\subset \mathcal F^\gamma_{[0,1]}\subset [0,1]\times X$ connecting  
$\{0\}\times \textup{Fix}(f^\gamma_0)$ and
$\{1\}\times \textup{Fix}(f^\gamma_1)$. Thus, the set
$$\C:=\{ (\gamma(t),x) : (t,x)\in \C_\gamma \}\subset\mathcal F_A \subset A\times X $$
connects $\mathcal F_\a$  and $\mathcal F_\b$.

It is worth  mentioning that, although $\gamma$ is a continuous curve, the set $\C$ may not be an arc. Interestingly, it was shown in \cite{genericity} that if $E=\R^n$ and $f$ is smooth, then $\C$ is generically diffeomorphic to $[0,1]$: 
in more precise terms, if one considers the set  
$\mathcal S $ of $C^2$ functions $f:[0,1]\times X\to X$ endowed with the $C^1$ norm, then 
there exists an open dense set $\mathcal S_0\subset \mathcal S$ such that, for any $f\in \mathcal S_0$,  each connected component $\C$ intersecting $\mathcal F_0$ is an arc. 

The original version of Browder's theorem has encountered applications in diverse fields, such as the nonlinear complementary theorem in programming theory \cite{E1}
and game theory \cite{G1}. 
The infinite-dimensional 
version expressed by Theorem \ref{main} with $A=[a,b]$ was employed in \cite{shaw}   
to prove the existence of solutions for an elliptic problem at 
resonance. Here, it is worth mentioning that no proof is given; instead, the author invokes the 
foundational work of Leray and Schauder \cite{LS}. 
Although it is true that Theorem \ref{main}  can be 
obtained by adapting the results in the latter paper (see also   
\cite{maw}), the invariance 
of the mappings $f_t$ allows  a considerably shorter treatment, as  shown below.  

A more subtle question arises on the following observation. 
When $A=[0,1]$, the original result can be   equivalently formulated as: 
 
\begin{thm}
Let $E$ be a normed space and let 
$X\subset E$ be bounded, closed and convex. Consider a continuous mapping $f:[0,1]\times X\to X$such that  $\overline{\textup{Im}(f)}$ is compact.  
Then there exists   
a connected set $\C\subset \mathcal F_{[0,1]}$ such that the projection $\pi:\C\to [0,1]$ defined by $\pi(t,x):=t$ is onto.  
\end{thm}

The equivalence follows immediately from the fact that
the only connected subsets of $\R$ are the intervals; however, it is not clear whether or not the latter result may be extended for a given arcwise connected set $A$. Observe, for instance, that if one considers sets $\C_{\a,\b}$ as in Theorem \ref{main} for all $\a,\b\in A$, then the set
$$\C:=\bigcup_{\a,\b\in A} \C_{\a,\b}
$$
may not be connected. 
As pointed out in \cite{soso}, an easy extension is obtained when $A$ is a Peano space, namely, when there exists a continuous surjective curve $\gamma:[0,1]\to A$. 
These spaces are characterized by the Hahn-Mazurkiewicz theorem and include, in particular, all those sets $A$ that are homeomorphic to a compact convex subset of a normed space. The main result in \cite{soso} extends the result to the case in which $A$ 
is a connected compact Hausdorff space, but not necessarily locally connected. However, the 
compactness assumption may be dropped 
in the specific case, frequent  
in applications, in which Fix$(f_\a)$  is a singleton for some
$\a\in A$. This is due to the obvious fact that, in such situation, the set
$$\C:= \bigcup_{\b\in A} \C_{\a,\b}
$$
is indeed connected since the intersection of all the  sets $\C_{\a,\b}$ is nonempty. 

\begin{cor}\label{coro}
In the situation of Theorem \ref{main}, assume there exists $\a\in A$ such that $\textup{Fix}(f_\a)
=\{ x\}$ for some $x\in X$. Then there exists   
a connected set $\C\subset \mathcal F_A$ such that the projection $\pi:\C\to A$ defined by $\pi(\b,x)=\b$ is onto.  
\end{cor}

\section{Short proof of Theorem \ref{main} with $A=[0,1]$}

By the Dugundji extension theorem \cite{D}, there exists a compact 
extension of $f$ 
to the whole space $E$ with range in $X$. Thus, we may  assume  $f:[0,1]\times E\to X$. 
It is clear that the sets Fix$(f_t)$ remain the same, because all the  fixed points of the extended function lie in $X$.  
Define
as before $\mathcal F:=\mathcal F_{[0,1]}$, 
which is  a compact subset of $[0,1]\times E$ 
because Fix$(f_t)\subset \overline{\textup{Im}(f)}$ is closed for all $t$. 
Suppose the result is false, then by  
Whyburn's Lemma \cite{why}, there exists
an open bounded set  
$U\subset [0,1]\times E$ containing 
 $\mathcal F_0$  and disjoint with $\mathcal F_1$ such that $\partial U\cap \mathcal F =\emptyset$. The latter property 
simply says, for all $t$, that $f_t$ has no fixed points $x$ satisfying $(t,x)\in \partial U$. Thus, the homotopy invariance of the Leray-Schauder degree applies, namely,   $\deg_{LS}(I-f_t,U_t,0)$ does not depend on $t\in[0,1]$, where
$$U_t:= \{ x:(t,x)\in U\}.
$$
Next, take $R>0$ sufficiently large such that   $X\cup U_0\subset B_R(0)$, then 
$I-\lambda f_0$ does not vanish 
on $\partial B_R$ for $\lambda\in [0,1]$. 
Hence, because 
 $ \textup{Fix} (f_0)\subset U_0$, 
the excision property of the degree and the homotopy invariance imply that  
$$\deg_{LS}(I-f_0,U_0, 0)=\deg_{LS}(I-f_0,B_R(0), 0) = 
\deg_{LS}(I,B_R(0), 0) =1.
$$
On the other hand, since $\mathcal F_1\cap\overline U=\emptyset$, it follows that $\textup{Fix}(f_1)\cap \overline {U_1} =\emptyset$; thus, 
$$\deg_{LS}(I-f_1,U_1, 0)=0,$$
which contradicts the fact that the degree is constant over the homotopy. 
\hfill{$\Box$}
\begin{rem}\label{schau}
The preceding result can be proven by means of Schauder's theorem only. 
Indeed, let us give a slightly simplified version of the argument presented 
in \cite{soso2}. In the situation of the previous proof, recall that the Whyburn lemma also ensures the existence 
of disjoint compact sets  $K_0$ and $K_1$ containing 
respectively $\mathcal F_0$ and $\mathcal F_1$ and such that
$\mathcal F=K_0\cup K_1$. 
 Next, define a continuous Urysohn function $\mu:[0,1]\times X\to [0,1]$ such that 
$$\mu|_{K_j}\equiv j, \qquad j=0, 1
$$
and the compact operator $F:[0,1]\times X\to [0,1]\times X$ given by
$$F(t,x)= (1-\mu(t,x), f(t,x)).
$$
From Schauder's theorem, $F$ has a fixed point $(t,x)\in \mathcal F$, whence $\mu(t,x)=0$ or $\mu(t,x)=1$. This implies, respectively, that $t=1$ or $t=0$, which contradicts the fact that $\mathcal F_j\subset K_j$. 

\end{rem}

\section{Applications:  from known results to open pro-blems} 
 \label{applic}

In this section, we shall present  applications of the Browder theorem to several boundary value problems. In what follows, some new results shall be established, as well as novel viewpoints of known results. 

In the first place, we shall consider a nonlocal boundary value problem, for which   
both the results and the approach are, to the author's knowledge, new. 
The second application concerns   pendulum-like equations and more general second-order resonant problems. 
We shall start with very  well 
known results from \cite{castro} and \cite{four-maw}, with the aim of showing that  
the Browder theorem can be also seen as an efficient tool for shortening proofs. 
In particular, the method of non-well ordered upper and lower solutions can be understood as an immediate consequence of the existence of a \textit{continuum} of  fixed points of an appropriate operator. Moreover, we shall unveil a connection between 
the ideas in the present paper and the celebrated
Landesman-Lazer theorem, which seems to be unexplored in the literature. 
As it  shall be pointed out, some new results 
are deduced as well; specifically, those related to 
the range of resonant semilinear operators for systems and its topological properties. 
Finally, Subsection \ref{applic}.\ref{chemost} is devoted to a delayed chemostat model previously studied by the author. Here, the goal is to give a more concise approach in terms of the Browder theorem. 
Furthermore, an open question shall be posed, regarding the possibility of obtaining an alternative proof, based only on a Schauder fixed point argument.

\begin{enumerate}

\item \textbf{A nonlocal boundary value problem}

\label{nonloc} 
Let $\Omega=(-L,L)\subset \R$ be a bounded interval and  consider the following  problem 
for some continuous functions $f:\overline\Omega\times\R\to \R$ and $g:\R\to\R$:
\begin{equation}
\label{nonlocal}
\left\{\begin{array}{cc}
u''(t) = f(t,u(t))     &  t\in \Omega  \\
u(\pm L)= g(u(0)).     & {}
\end{array}\right.
\end{equation}
Here, the nonlocal boundary condition may be regarded as a nonlinear instance of the three-point boundary condition studied in \cite{sun}. For a more general multi-point boundary condition, see e.g. \cite{graef}.

An easy application of Theorem \ref{main}
can be implemented as follows. Observe that, if $u$ is a solution, 
then $u|_{\partial \Omega} = g(u(0)):=c$ and the function 
$v:=u-c$ is a solution of the problem 
 \begin{equation}
\label{homog} v''(t)= f(t,v(t)+c),\qquad v|_{\partial \Omega=0}
 \end{equation}
 satisfying
 \begin{equation}
 \label{bc} g(v(0)+c)= c.
     \end{equation}
Conversely, if $v$ is a solution of (\ref{homog}) for some $c\in \R$ such that (\ref{bc}) is fulfilled, then $u:=v+c$ solves the original problem (\ref{nonlocal}). Thus, we may consider the space $E:=\{v\in C(\overline\Omega) : v|_{\partial\Omega}=0\}$ and the  operator 
$T:\R\times E\to E$ given by $T(c,w):=v$, where $v$ is the unique element of $E$ that solves the linear equation $ v''(t)=f(t,w(t)+c)$. 
Recalling the standard estimate
$$\|v\|_\infty \le k\| v''\|_\infty\qquad v\in E,
$$
it is immediately seen that $T$ is compact. 
In order to verify the assumptions of Theorem \ref{main}, we need to find a closed convex set $X\subset E$ such that $T_c(X)\subset X$ for all $c$. Assume for simplicity that  $f$ is bounded, then
$$\|T(c,w)\|_\infty \le k\|f(\cdot,w+c)\|_\infty\le k\|f\|_\infty:=R
$$
and we conclude that $X:=\overline {B_R(0)}$ is $T_c$-invariant for all $c$. 
From Theorem \ref{main}, for each $a<b$  
there exists a subset 
$$\C\subset  \bigcup_{a\le c\le b} \{c\}\times \textup{Fix}(T_c) $$ 
connecting $\{a\}\times \textup{Fix}(T_a) $ with 
$\{b\}\times \textup{Fix}(T_b)$. Furthermore, the function
$$\Phi(c,w):=   c- g(w(0)+c)
$$
is continuous; thus, in order to solve the original problem, it suffices to find $a, b\in\R$ such that $\Phi$ takes different signs when restricted to  
$\{a\}\times \textup{Fix}(T_a)$ and $\{b\}\times \textup{Fix}(T_b)$. 
A trivial occurrence  of such situation is when $g$ is sublinear or, more generally, when
\begin{equation}
    \label{sublin}\limsup_{|u|\to\infty} \left|\frac {g(u)}u\right| <1. 
\end{equation}
It might be argued that there is no need to invoke Browder's theorem here, because the proof is readily deduced from a direct fixed point argument, by defining instead the compact operator $\tilde T:C(\overline\Omega)\to C(\overline\Omega)$ by $\tilde T(v):=u$, the unique solution of the Dirichlet problem
$$u''(t)=f(t,v(t)),\qquad u(\pm L) =g(v(0)).
$$
Indeed, since there exist $A<1$ and $B>0$ such that $|g(v)|\le A|v|+B$, the fact that
$$\|Tv-g(v(0))\|_\infty \le R,
$$
implies
$$\|Tv\|_\infty \le A|v(0)|+ B + R \le A\|v\|_\infty + B + R.
$$
Thus, taking $M:=\frac {B+R}{1-A}$, it is verified that $T(\overline {B_M(0)})\subset \overline {B_M(0)}$ and Schauder's theorem applies. 
 
 The situation is different when $g$
is superlinear or, more generally,  when
\begin{equation}\label{superlin}
 \liminf_{|u|\to\infty} \frac {g(u)}u >1.
\end{equation}
Here, a nontrivial bounded invariant region for $\tilde T$ may not exist. 
This is comparable with the existence of fixed points for expansive operators: for instance, the only invariant region for the real function  
$f(x):=2x$ is the set composed by the (unique) fixed point of $f$.
Although it is still possible to deduce the existence of solutions by computing the Leray-Schauder degree of $I-\tilde T$, the previous setting in terms  of Theorem \ref{main} allows to reduce the proof to a single line, since it is clear that $\Phi(c,v)<0<\Phi(-c,v)$ for all $v\in\overline {B_R(0)}$ and $c\gg 0$. 
 A more general non-asymptotic assumption reads 
  \begin{equation}\label{non-asym}
  g(a+r) \le a,\qquad g(b+r)\ge b
   \end{equation}
 for some $a, b\in \R$ and $|r|\le R$, which clearly implies 
$\Phi|_{\mathcal F_a}\le 0\le \Phi|_{\mathcal F_b}$. For example, if $g(u)= u + \sin(\sqrt[3]{u})$, then the problem has infinitely many solutions for arbitrary bounded $f$, although $\lim_{|u|\to\infty}\frac{g(u)}u=1$.

\begin{rem}
When $g$ satisfies (\ref{sublin}),  the assumption that $f$ is bounded may be relaxed: for example, it suffices to assume that $f$ is sublinear or, more specifically, that 
$$|f(t,u)| \le \ee |u| + C
$$
for some $C>0$ and $\ee>0$ small enough.

Indeed,  fix $\ee>0$ such that  $k\ee < 1 - \limsup_{|u|\to \infty} \left|\frac {g(u)}u\right|$ and observe that, if $u$ is a solution of the problem with $u(\pm L)=c$, 
then  
$$\|u\|_\infty \le k(\ee \|u\|_\infty + C) + |c|.
$$
As before, set $A,B>0$ with $k\ee + A<1$
such that $|g(v)|\le A|v| + B$, then, because $g(u(0))=c$, 
$$ \|u\|_\infty \le k(\ee \|u\|_\infty + C) + A|u(0)| + B,
$$
whence 
$$\|u\|_\infty\le \frac {kC+B}{1 - k\ee - A}:=R.
$$
Since the latter constant does not depend on $f$, the proof follows from a truncation argument. 

The same conclusion holds if $g$  satisfies (\ref{superlin}), although
the proof is more tricky. We may proceed as follows: 
in the first place, there exists a constant (still denoted $k$) such that if $u$ is a solution then
$$\|u-u(0)\|_\infty\le k\|u''\|_\infty \le k\ee\|u\|_\infty + kC.
$$
This, in turn, implies
$$\|u\|_\infty \le k\ee\|u\|_\infty + kC + |u(0)|.
$$
Next, fix $A>1$ and $B>0$ such that $|g(u)|\ge A|u| - B$, then
$$A|u(0)| \le |g(u(0))| + B = |u(\pm L)| + B\le\|u\|_\infty + B,
$$
that is, 
$$\left(1 - k\ee - \frac 1A
\right)\|u\|_\infty \le \frac BA +C.
$$
In other words, it suffices to take $\ee$ such that $(1-k\ee)\liminf_{|u|\to\infty} \frac {g(u)}u >1$.
\end{rem}

A less elementary matter  arises when dealing with systems rather than a scalar equation. Observe, indeed, 
that if $f:\overline \Omega\times \R^N\to \R^N$ is bounded, 
then the previous machinery  
can be arranged
 exactly in the same 
way as before and, by the ``strong" version of the Browder theorem, for an arbitrary closed ball $\overline B\subset\R^N$
there exists a connected set 
$$\C\subset \bigcup_{c\in \overline B} \{c\}\times \textup{Fix}(T_c) $$
whose projection to $\overline B$ is onto.  
In spite of that, it is not clear if  the $N$-dimensional analogue of the preceding
function $\Phi$  vanishes at some $(c,w)$ because, in principle, $\Phi(\C)$ may not be simply connected. Solutions are  obtained as fixed points of the previous operator $\tilde T$, both in the sublinear and the superlinear cases, namely
$$\limsup_{|u|\to\infty} \frac {|g(u)|}{|u|}  <1
$$
and
$$\liminf_{|u|\to\infty} \frac {\langle g(u),u\rangle }{|u|^2}  >1
$$
respectively, where $|\cdot|$ stands now for the norm of $\R^N$. More generally, let us observe
that the non-asymptotic  condition (\ref{non-asym}) has an $N$-dimensional (strict) analogue for systems:

\begin{prop} In the previous setting, assume there exists $D\subset \R^N$ open such that:
\begin{enumerate}
    \item $g(r+c)\ne c$ for every $c\in \partial D$ and $|r|\le R$. 
    \item $\deg(I-g,D,0)\ne 0$. 
\end{enumerate}
Then system (\ref{nonlocal}) admits at least one solution $u$ with $u|_{\partial\Omega}\in D$. 

\end{prop}
\begin{proof}
Define the  operator 
$\mathcal T_\lambda:\overline D\times \overline {B_R(0)}\to \R^N\times \overline{B_R(0)}$ given by
$$\mathcal T_\lambda(c,w):= (g(w(0)+c), \lambda T(c,w)).
$$
If $\mathcal T_1$ has a fixed point on the boundary, then there is nothing to prove. Otherwise, suppose  that  $\mathcal T_\lambda(c,v)=(c,v)$ for some $(c,v)\in \partial (D\times B_R(0))$ and $0\le \lambda< 1$, then $c\in D$ and $R=\|v\|_\infty =\lambda \|T(c,v)\|_\infty <R$, a contradiction. This implies
$$\deg(I - \mathcal T_1, D\times B_R(0),0) =\deg(I - \mathcal T_0, D\times B_R(0),0) = \deg(I-g,D,0),
$$
and the result follows. 
\end{proof}

Besides the preceding degree argument, it would be interesting to
analyse if a proof is possible by means of Theorem \ref{main} only. 

\begin{rem}
The problem treated in the present section may be 
considered as an ODE analogue of the more general problem studied in \cite{gennaro}. Here, not only the boundary condition but also the equation may contain nonlocal terms. 
An example of a nonlocal ODE with local boundary conditions is introduced 
in \cite{thom}, arising on a two ion electro-diffusion model. In \cite{AKR}, a two-dimensional  shooting technique is developed in order to prove the existence of solutions;
however, it is not clear whether or not the result can be obtained by means of a Browder type argument only.

\end{rem}

 \item \textbf{Range of some semilinear operators at resonance}
 \label{pendulum}
 
 In a celebrated paper from 1980, A. Castro \cite{castro} posed the following question: 
 for which $\omega$-periodic forcing terms $p(t)$ the pendulum equation
 $$u''(t) + \sin u(t) =p(t)
 $$
 admits an $\omega$-periodic solution? 
 The problem is called \textit{resonant}, because
the kernel of the 
associated linear operator $\mathcal Lu:=u''$ over the space of $\omega$-periodic solutions is 
nontrivial. Together with the fact that 
the nonlinear term $g(u):=\sin u$ is bounded, this implies that the problem has no $\omega$-periodic solutions for some choices of $p$: indeed, taking average at both sides of  the equation yields the necessary condition $|\overline p|\le 1$. 
However, this condition is not sufficient, 
so the answer to the previous question is far from being obvious. Using variational methods, it is easy to prove that the equation admits at least one $\omega$-periodic solution when $\overline p=0$, because the ($\omega$-periodic) associated functional achieves a minimum. In view of this, it proves convenient to write $p=p_0 + s$, where $p_0$ has zero average; thus, the problem of describing 
the range of the semilinear 
operator $S u:= u'' + \sin u$ is reduced to find, for each $p_0$, the set $I(p_0)$ of values $s\in [-1,1]$ such that an $\omega$-periodic solution of the equation exists. If for simplicity we 
consider $S:D\subset C_\omega\to C_\omega$, where 
$C_\omega$ denotes the space of $\omega$-periodic continuous functions and $D:=C_\omega\cap C^2$, then the results in \cite{castro} 
for $\omega\le 2\pi$ imply that, for each $p_0\in \widetilde C_\omega:=\{ \varphi\in C_\omega:\overline\varphi=0\}$ the set $I(p_0)$ is a  compact interval containing the origin, whose endpoints depend  continuously on $p_0$.  
Thus, identifying $C_\omega$ with $\widetilde C_\omega\times \R$, 
the range of $S$ can be expressed as
$$\textup{Im}(S)= 
\bigcup_{p_0\in \widetilde C_\omega} \{p_0\} \times I(p_0).
$$
The proofs in \cite{castro} rely on the variational 
structure of the problem; a few years later, Fournier 
and Mawhin 
\cite{four-maw} extended the results 
for  arbitrary periods and assuming that   the equation contains a friction term $au'$, so the problem is 
non-variational when $a\ne 0$. 
Here, we shall analyse the method used in \cite{four-maw} through the lens of the Browder theorem. 
It is worthy to mention that the setting applies for more general problems such as
\begin{equation}\label{nonvar} 
u''(t) + au'(t) + g(u(t)) = p_0(t) + s
    \end{equation}
with $g$ continuous and bounded, which has been treated for example in \cite{ht}. 
With this in mind, let us observe 
that the periodic problem for (\ref{nonvar}) can be   
regarded as 
nonlocal, in the following sense. 
Because  $p_0$
is $\omega$-periodic, if $u$ is a 
solution of (\ref{nonvar}), then $u$ is $\omega$-periodic if and only if 
$$u(0)=u(\omega),\qquad u'(0)=u'(\omega)
$$
or, since $\overline {p_0}=0$, if and only if
$$u(0)=u(\omega), \qquad \overline{g(u)} = s.
$$
As it may be recalled, this procedure 
 is the essence of the Lyapunov-Schmidt decomposition. 
Again, a compact operator is defined by solving 
a homogeneous Dirichlet problem; namely, for $(c,w)\in \R\times C_\omega$, let $v:=T(c,w)$ be defined as the unique solution of the linear problem
\begin{equation}\label{dir-hom}
v''(t) + av'(t) = p_0(t) - g(c+w(t)) + \overline {g(c+w)}, \qquad v(0)=v(\omega)=0.
\end{equation}
Upon integration, it follows that $v\in C_\omega$, which implies 
Im$(T)\subset C_\omega$. Hence, 
$v$ is a fixed point of $T_c$ if and only if $u:=c+v$ is an $\omega$-periodic solution of the problem for $s=\overline{g(u)}$. 
Furthermore, using again standard estimates we obtain
$$\|v\|_\infty \le k \|v'' + av'\|_\infty\le R
$$
for some constant $R$, that is, 
$T(\R\times \overline{B_R(0)})\subset \overline{B_R(0)}$. 
Thus, if we define as before
$$\mathcal F_\R =  \bigcup_{c\in\R} \{c\}\times \textup{Fix}(T_c)
$$
then the set $I(p_0)$ coincides with $\I(\mathcal F_\R)$, where the continuous  
mapping
$\I:\R\times C_\omega\to \R$ given by 
$\I(c,w):=\overline {g(c+w)}$. It is clear 
that the set $I(p_0)$ is contained in the interval $[-\|g\|_\infty,\|g\|_\infty]$, although 
its compactness is not guaranteed unless we assume extra conditions on $g$. One of these possible extra conditions is that, as in 
the pendulum equation, $g$ is $\sigma$-periodic for some $\sigma>0$, which allows to restrict the mapping $\I$ 
to the compact set $\mathcal F_{[0,\sigma]}$. The same happens when \textit{a priori} bounds 
exist: if one knows that all the possible $\omega$-periodic solutions of the problem   are bounded by 
a constant $M$ depending only on $\|p_0\|_\infty$, then $\I$ can be restricted to the set $\mathcal F_{[-M,M]}$.

A typical situation in which compactness fails is the Landesman-Lazer case, that is, when the limits
$$g_{\pm}:= \lim_{u\to \pm \infty} g(u)
$$
exist and are distinct. 
If furthermore   
$g(u)\ne g_\pm$ for all   
$u\in \R$, then it is well known that $I(p_0)$ is the (open) interval of all those values 
lying strictly between $g_-$ and $g_+$. 
Perhaps we don't need a new proof of this but, again, Browder's theorem reduces it to an easy exercise. To fix ideas, suppose that $g_-<g_+$ and 
define as before a mapping $T(c,w)$ from solving the 
linear problem (\ref{dir-hom}) and, for arbitrary 
$s\in (g_-,g_+)$ define the 
continuous function $\Phi(c,w):= s- \overline{g(w+c)}$. Take $M$ sufficiently large such that 
$$g(u) > s > g(-u)\qquad u\ge M,
$$
then, for $w\in \overline{B_R(0)}$ we deduce, for $c:=R+M$, that 
$w(t)+ c\ge M$ and $w-c\le -M$. Thus, if $\mathcal C\subset \mathcal F_{[-c,c]}$ connects $\mathcal F_{-c}$ and $\mathcal F_c$, the mapping $\Phi$ changes sign over $\mathcal C$ and, consequently, vanishes at some point. 
Conversely, 
when $s\ge g_+$ or $s\le g_-$, the non-existence of solutions 
is deduced simply by integrating the equation.

Also, we cannot ensure that $0\in I(p_0)$: 
this may be false even for $g(u)=\sin u$, provided that $a\ne 0$. This was shown in the early work \cite{rafa-pend}, 
see e.g.  \cite{maw-75} for a complete survey of the problem. 
Due to Schauder's theorem, the mapping $T_c$ has fixed points for all $c$; thus  $I(p_0)$ is 
nonempty although, in general, it cannot be proven that it contains more than one point. This is one of the most famous  
open problems concerning the pendulum equation: 
decide whether or not there exists $p_0$ such that $I(p_0)$ is a singleton. Such a situation is called \textit{degenerate} and it was demonstrated that, generically, it does 
not occur, that is: except possibly for a meager subset of $\widetilde C_\omega$. Observe, incidentally, that 
if $I(p_0)=\{s\}$ then all the solutions of the problem
$$u''(t)+au'(t) + \sin(u(t)) =
p_0(t)+\overline{\sin(u)}, \qquad u(0)=u(\omega)=c
$$
for arbitrary $c$ satisfy $\overline{\sin(u)}=s$, so we can deduce the existence of  a \textit{continuum} $\C$ of such solutions $u_c\in C_\omega$ with $c\in\R$. Differently from what occurs in the general abstract case, here $\C$ is indeed a curve, because, as shown 
by Ortega and Tarallo in \cite{rafa-cont}, 
 $u_c$ is unique and the map $c\mapsto u_c$ is continuous. 
Amazingly, 
this can be seen, once more, as 
a consequence of 
Theorem \ref{main}. To this end,  following  
the ideas in \cite{rafa-cont}, let us firstly prove 
that different 
solutions do not intersect at any point.  
Suppose that $u$ and $v$ are solutions such that $v(t_0)>u(t_0):=\eta$ at some $t_0$. Without loss of generality, we may assume 
$t_0=0$ and take, for $a< \eta-2R$, 
a subset $\C\subset \mathcal F_{[a,\eta]}$ that connects $\mathcal F_{a}$ and $\mathcal F_\eta$. 
Let 
$\hat\C\subset C_\omega$ be the projection of $\C$, 
then $v\notin \hat\C$ and
observe that if $v$ intersects some $w\in \hat \C$ at some $t$, 
then the uniqueness of the initial value problem implies 
$v'(t)\ne w'(t)$. In turn, this implies that the set 
$$A_v:=\{w\in\hat\C: \hbox{$w$ intersects $v$}\}
$$
is open and closed; in other words, $v$ does not intersect any element of $\hat \C$. 
Now suppose  $u\notin\hat \C$ and
observe that, as before, the set  $A_u$ is open and closed 
in $\hat\C$. 
Moreover, $u$ intersects some element of 
$\hat\C$ at $t=0$, so $A_u=\hat\C$.  
But, on the other hand, there exists $w\in \hat \C$ such that $w(0)=a$
which satisfies, for all $t$,
$$w(t) < a + R < \eta - R < u(t),
$$
a contradiction. Thus, $u\in\hat \C$ and, consequently, does not intersect $v$. The continuity of the map $c\mapsto u_c$ follows now straightforwardly.

In contrast with the problem of degeneracy, proving 
that $I(p_0)$ is always an interval is easy and, again, relies on the existence of a \textit{continuum}.
Indeed, 
it suffices to 
prove that  
if $s_2<s_1$ are two elements of $I(p_0)$, 
then $s\in I(p_0)$ for any $s\in (s_2,s_1)$.
Take $(c_j,u_j)\in \mathcal F$ preimages of $s_j$ for $j=1,2$,
then 
$$u_1''(t) + au_1'(t) + g(u_1(t)) > p_0(t) + s > u_2''(t) + au_2'(t) + g(u_2(t)),
$$
that is, $u_1$ and $u_2$ are, respectively, a lower and an upper solution of the problem. 
The issue here is that $u_1$ and $u_2$ are not necessarily well-ordered, although the existence of a solution is still verified because $g$ is bounded. 
This is  a well known fact, whose proof  becomes almost effortless with the help of Theorem \ref{main}. To this end,
 fix a constant $M$ such that 
 $$M > \max\{ \|u_1\|_\infty, \|u_2\|_\infty \}+ R.$$
Since any solution $u$ satisfies $\|u-u(0)\|_\infty\le R$, it is  deduced, for $u(0)=\pm M$, that 
$\max\{ \|u_1\|_\infty, \|u_2\|_\infty \} < M - |u(t)-u(0)|$  for all $t$, that is
$$u_1(t), u_2(t) < u(t)\qquad \hbox{if $u(0)=M$},
$$ 
$$u_1(t), u_2(t) > u(t)\qquad \hbox{if $u(0)=-M$}.
$$  
 Next,  consider a set $\C\subset \mathcal F$ connecting 
$\mathcal F_{-M}$ and $\mathcal F_{M}$. 
If $\I$ takes the value $s$  at some point in $\C$, then we are done, otherwise $\I -s$ has constant sign, say 
e.g. $\I|_{\C}>s$. 
Then, we may take an element  $(-M,u)\in \C$ and  observe that, because $\overline{g(u)}=\I(-M,u)>s$, the function 
$u$ is a lower solution of the problem for $s$ and, consequently, 
$(u,u_2)$ is a well-ordered couple of a lower and an upper solution.  If we assume, instead, that  $\I|_{\C}<s$, then taking $(M,v)\in \C$ produces the well-ordered couple $(u_1,v)$.

Next, let us consider the problem of establishing the continuity of the   mapping $p_0\mapsto I(p_0)$. As mentioned, for the pendulum equation this was proven in \cite{castro} 
for the variational case and later extended to
the non-variational 
case by means of topological arguments (see  \cite{maw-75}). 
However, the subject was not tackled in \cite{ht} 
for arbitrary bounded $g$, although
the proof can be also performed in a very simple and direct way as follows. Assume that $p_n\in C_\omega$ has zero average for all $n$ and $p_n\to p_0$ uniformly. In order to prove that $I(p_n)\to I(p_0)$,
we shall proceed in two steps. 
Firstly, 
let 
$b_n:= \sup{ I(p_n)}$   and consider an arbitrary subsequence,  
still denoted $\{b_n\}$, converging to some value $b$. Fix $s<b$, then there exist 
  $\tilde b_{n}\in I(p_{n})$ such that 
  $s< \tilde b_{n}-\ee$ for some $\ee>0$. 
 Taking  $u_{n}\in C_\omega$ solutions for $p_{n} + \tilde b_{n}$, we obtain;
 $$u_{n}''(t) + au_{n}'(t) + g(u_{n}(t)) =p_0(t) + s + p_{n}(t) -p_0(t) + \tilde b_{n} - s > p_0(t)+s
 $$
for $n\gg 0$, that is, $u_{n}$ is a lower solution  of the problem for $p_0+s$. 
In the same way, if we take an arbitrary subsequence
of $a_n:= \inf{ I(p_n)}$ converging to
some $a<s$, then a lower solution for $p_0+s$ is obtained and, as before, we conclude that a solution exists when $a<s<b$. 
Because the subsequences are arbitrary, this already proves that 
$$(\liminf_{n\to\infty} a_n, \limsup_{n\to\infty} b_n)\subset I(p_0).
$$
In the second place, 
suppose that 
$s\in I(p_0)$ satisfies $s>b$ and fix $\ee>0$
such that
 $s>b_{n}+2\ee$ for all $n$. Let 
 $u\in C_\omega$ be a solution for $p_0+s$, 
  then it follows as before that
  $u$ is a lower solution for the problem corresponding to $p:=p_{n}+  b_{n} + \ee$, provided that 
    $n\gg 0$. 
     On the other hand, setting
       $\tilde b_{n}\in I(p_{n})$  and
     a solution $u_{n}\in C_\omega$  for 
    $p_{n}+ \tilde b_{n}$, then 
 it is clear that $u_{n}$ is an upper solution for $p$. This implies that $b_{n}+\ee\in I(p_{n})$, a contradiction. 
  An analogous argument holds if $s<a$ and, again, due to the fact that  the subsequences are arbitrary we conclude that 
    $$\overline{I(p_0)} \subset [\limsup_{n\to\infty} a_n, \liminf_{n\to\infty} b_n].$$ 
By combining the two steps, it is deduced that 
$$(\liminf_{n\to\infty} a_n, \limsup_{n\to\infty} b_n)
\subset [\limsup_{n\to\infty} a_n, \liminf_{n\to\infty} b_n];$$
summarizing, the sequences $\{a_n\}$ and $\{b_n\}$ converge respectively to the lower and upper endpoints of $I(p_0)$.

It is worth recalling, in pendulum-like equations, 
that the periodicity 
of $g$ also implies, when  $s$ is an interior point of $I(p_0)$, the 
existence of a second solution which is {geometrically different}, in the sense  
that it does not differ with the first one by an integer multiple of the period. 
This is a consequence of the so-called \textit{three solutions theorem}; we sketch a proof for the sake of completeness. Assume that $g(t+\sigma)\equiv g(t)$ and set  solutions $u_1$,  $u_2$  corresponding to some $s_2 <s_1$. Fix  $s\in (s_2,s_1)$, then there exists $k\in \Z$ such that 
$$u_1(t)< u_2(t) +  k\sigma, \qquad u_1(t)+\sigma\not< u_2(t) + k\sigma
$$
and the excision property of the Leray-Schauder degree
shall be applied  to the sets
$$\Omega_1:= \{u\in C_\omega:  u_1(t)< u_2(t) + k\sigma\},$$
$$
\Omega_2:= \{u\in C_\omega:  u_1(t)+\sigma< u_2(t) + (k+1)\sigma\},
$$
both contained in 
$$
\Omega:= \{u\in C_\omega:  u_1(t)< u_2(t) + (k+1)\sigma\}.
$$
To this end, as usual,  we may define  
$K:C_\omega\to C_\omega$ given by $Kw:=v$, the unique $\omega$-periodic solution of the (non-resonant) problem
$$v''(t)+av'(t) - v(t)= p_0(t) + s - g(w(t)) -  w(t).
$$
A standard result in the theory of strict well-ordered upper and lower solutions ensures  
that the degree  of $I-K$ over $\Omega_1$, $\Omega_2$ and $\Omega$ is equal to $1$; hence, there exist   solutions $w_j\in \Omega_j$
and $w\in \Omega\backslash (\Omega_1\cup\Omega_2)$ and, clearly,   at least two of  
the three solutions $w, w_1$ and $w_2$ are geometrically distinct.

As in our first example, we may also make
some considerations  about the non-scalar 
case.  The construction can be done exactly in the same way as before, and the existence of a nonempty bounded set 
$I(p_0)=\textup{Im}(\I)\subset \R^N$ follows. 
Using the mean value theorem for vector integrals, it is readily verified that $I(p_0)$ is contained in the convex hull of Im$(g)$; moreover, $I(p_0)$ is compact if, for example, $g$ satisfies some periodicity assumption, e. g.
\begin{equation}\label{per-sys}
g(u+ \sigma_je_j) \equiv g(u),
\end{equation}
where $\sigma_j>0$ and $\{e_1,\ldots,e_N\}$ is a basis of $\R^N$.
But an issue occurs when one tries 
to prove the connectedness, because 
the method of well-ordered upper and lower 
solutions for systems  usually requires stronger conditions, e.g.  
$\a_j\le \b_j$ and
$$
\a_j''(t) + a \a_j'(t) + g_j(\hat \a(t)) \ge p(t) 
\ge \b_j''(t) + a \b_j'(t) + g_j(\hat \b(t))
$$
for $j=1,\ldots, N$, where 
$\hat\a_j=\a_j$, $\hat \beta_j=\b_j$ and $\a_k\le \hat \a_k, \hat \b_k\le \b_k$ for all $k$. 
These conditions might be of some help   when dealing with weakly coupled systems but, 
in general, 
the problem of determining the shape of $I(p_0)$, or even if 
it is connected,  is open.

An alternative approach can be
employed if one assumes the same assumption imposed in \cite{castro} for the scalar case, namely the relaxed monotonicity condition
\begin{equation}
\label{uniq}
\frac{\langle g(u)-g(v),u-v\rangle}{|u-v|^2} < \left(\frac {2\pi}\omega\right)^2 \qquad \forall\, u\ne v.\end{equation}
This case was 
already 
analysed in  \cite{Ku} for 
the  variational case, that is, with $a=0$ and $g=\nabla G$. However, the following extension was not previously
treated in the literature:

\begin{prop}
In the previous situation, assume that
 (\ref{uniq}) holds. Then  $I(p_0)$ is arcwise connected.

\end{prop}

\begin{proof}
Given 
$x\in\R^N$, let us prove the following claim: there exists  a unique  $\omega$-periodic function $w_x$  such that 
$w_x'' + aw_x' + g(x+w_x) - p_0 
$ 
is constant and $\overline{w_x}=0$. 

\textit{Existence}: We may  consider now the operator 
$$\mathcal T^x(c,w):= (x-\overline {T(c,v)},  T(c,v)),
$$
and prove as before that the degree of $I-\mathcal T^x$ over a large ball of $\R^N\times C_\omega$ is equal to $1$. This implies the existence of  
$v$ such that $v=T(c,v)$ and $\overline v +c=x$; thus,  $w_x:= c+v-x$ fulfills all the requirements.

\textit{Uniqueness}:  
Suppose that $w_1\ne w_2$ are solutions, then   $w:=w_1-w_2$ verifies
$$w''(t) + aw'(t) + g(x+w_1(t)) - g(x+w_2(t))= C
$$
for some $C\in\R^N$. 
Multiplying by $w$ and using  that $\overline w=0=w(\omega)-w(0)$, we get
$$\int_0^\omega |w'(t)|^2\, dt = \int_0^\omega \langle g(x+w_1(t)) - g(x+w_2(t)),w(t)\rangle\, dt $$
$$< \left(\frac {2\pi}\omega\right)^2\int_0^\omega |w(t)|^2\,dt,
$$
which contradicts  Wirtinger's inequality.

Furthermore, observe that  the mapping $x\mapsto w_x$ is continuous. 
Indeed, assume $x_n\to x$ and  consider the respective solutions 
$w_n:=w_{x_n}$ and $w:=w_x$. If $w_n\not\to w$, then passing to a subsequence  we may assume that 
$w_n$ converges to some $z\ne w$ for the $C^1$ norm and $w_n''$ converges weakly in $L^2$. Multiplying the equation by test functions, it is   readily seen that $z=w$, a contradiction.

We conclude that
the set $I(p_0)$ is characterized as the range of the continuous function of $\R^N$ defined by
$$x\mapsto  \overline{g(x+w_x)}
$$
and so completes the proof.
\end{proof}

Regarding the continuity of $I$ with respect to $p_0$, it is observed
that, as far as we do not have any information about the shape of the set $I(p_0)$, we need to use some 
notion of distance between bounded sets. The obvious choice is the Hausdorff metric, namely, given $K_1, K_2$ compact subsets of $\R^N$, the distance given by
$$d_H(K_1,K_2):= \max \{\sup_{y\in K_1}d(y,K_2), \sup_{y\in K_2}d(y,K_1)\}. 
$$
We shall prove that the mapping $I$ is continuous 
when the nonlinearity is periodic or of Landesman-Lazer type and leave the  general situation
for future research. 
Again, we recall that, up to the author's knowledge, only  the  variational case under condition (\ref{per-sys}) was 
previously treated in the referred work  \cite{Ku}.  

\begin{prop}
Assume that (\ref{uniq}) holds and the sequence $\{p_n\}\subset \widetilde C_\omega$ converges to $p_0$ uniformly. Assume, furthermore, that either $g$ satisfies (\ref{per-sys}) or has radial limits
$$g_v:=\lim_{r\to+\infty} g(rv)
$$
 uniformly for $v$ in the unit sphere $S^{N-1}\subset \R^N$. 
Then $\overline{I(p_n)}\to \overline{I(p_0)}$ for the Hausdorff metric.  
\end{prop}
\begin{proof}
Observe, in the first place, 
that the same argument employed before shows that the mapping 
$x\mapsto w_x$   also depends continuously on $p_0$, that is: if $(x_n,p_n)\to (x,p_0)$, then the corresponding sequence  $\{w_{x_n}^n\}$ converges to $w_x$. 
Let $s\in I(p_0)$, 
that is $s=\overline{g(x+w_x)}$ for some $x$. 
By dominated convergence, $\overline{g(x+w^n_x)}\to s$ which, in turn, implies that $d(s,\overline{I(p_n)})\to 0$. 
On the other hand, suppose that 
$d(s_n,I(p_0))\not\to 0$ for some sequence
$s_n=\overline{g(x_n+w^n_{x_n})}$, then passing to a 
subsequence we may assume $s_n\to s$, with $ s \notin\overline{I(p_0)}$ and $x_n\to L$, with $L=x\in\R^N$ or $L=\infty$.
In the first case,
it follows as before that $x_n+w_{x_n}^n\to x+w_x$ for some $x$; hence, 
$s\in I(p_0)$, a contradiction.  This already covers the case  in which (\ref{per-sys}) is satisfied. 
Thus, we may assume that 
$g$ have uniform radial limits and 
 $x_n\to\infty$. Taking a subsequence,  we may 
also assume that $\frac{x_n}{|x_n|}$ converges to some $v\in S^{N-1}$. Consequently, both sequences $\frac{x_n+w_{x_n}^n}{|x_n+w_{x_n}^n|}$ and 
$\frac{x_n+w_{x_n}}{|x_n+w_{x_n}|}$ converge uniformly to $v$. This implies that 
$s_n\to g_v\in \overline {I(p_0)}$, a contradiction. 
\end{proof}

Exactly the same conclusions can be obtained for the elliptic analogue    of the preceding problem, that is
$$\Delta u(x)+ \langle a,\nabla u(x)\rangle + g(u(x))=p(x)  \qquad x\in \Omega, $$
with the nonlocal condition
$$u|_{\partial \Omega}=c,\qquad \int_{\partial \Omega} \frac{\partial u}{\partial \nu} dS=0
$$
where $c$ is an undetermined constant and $\nu$ denotes the outer normal or, equivalently
$$u|_{\partial \Omega}=c,\qquad  \int_\Omega g(u(x))\, dx= \int_\Omega p(x)\, dx. 
$$

    \item \label{chemost} \textbf{A delayed chemostat model revisited}. In \cite{ars}, the following model was considered
    $$\left\{ \begin{array}{l}
         s'(t)= D(t)[s^0(t)-s(t)] -\mu(s(t))\frac {x(t)}\gamma  \\
          x'(t)=x(t)[\mu(s(t-\tau))- D(t)],
    \end{array}\right.
    $$
    where $D, s^0>0$ are continuous $\omega$-periodic functions, $\gamma>0$ is a constant and the mapping $\mu:[0,+\infty)\to [0,+\infty)$ is locally Lipschitz and strictly increasing, with $\mu(0)=0$. 
    It is assumed that $\tau>0$ is a fixed delay and, regarding the existence of $\omega$-periodic solutions,  we may also 
    assume without loss of generality that $\tau<\omega$.
    From the standard theorem of existence and uniqueness, for arbitrary initial values $\varphi\in C[-\tau,0]$ with $\varphi\ge 0$ and $x_0\ge 0$ there exists a unique solution defined on an interval $[-\tau,\delta)$ for some $\delta\in (0, +\infty]$. In particular, when $x_0>0$ it is seen  
    that $x(t)>0$ for all $t$ and, 
    if $s$ vanishes  at   
    some $t_0\in [0,\delta)$ then 
    $s'(t_0)>0$.
    This implies that $s(t)>0$ for $t>0$. Observe, moreover, that if $x_0=0$, then $x\equiv 0$ and $s$ coincides with the unique solution of the problem 
    \begin{equation}
        \label{v}
        v'(t)= D(t)(s^0(t)-v(t))
    \end{equation}
    satisfying   
$$ v(0)=\varphi(0).$$ 
It is readily verified that (\ref{v}) has a unique $\omega$-periodic solution $v^*$, which is positive, and we may call $\varphi^*:=v^*|_{[-\tau,0]}$. 
On the other hand, for $x_0>0$ and $\varphi\in X$, where
    $$X:=\{\varphi\in C[-\tau,0]: 0\le\varphi\le \varphi^*\}, $$
    the corresponding solution $s$ satisfies
    $s(t)<v^*(t)$ for all $t>0$. 
    Indeed, it is noticed that $w:=v^*-s$ satisfies
    $$w'(t) > -D(t)w(t),\qquad t>0
    $$
    and, upon integration, we obtain
    $$w(t) > e^{-\int_0^t D(s)\, ds}w(0)\ge 0.
    $$
As a consequence, all the trajectories $(s,x)$ with initial values 
    $\varphi\in X$ and $x_0\in [0,+\infty)$ are globally defined. Thus, we may consider the Poincar\'e operator
    $$P(\varphi,x_0):= (s_\omega,x(\omega)),
    $$
    where $(s,x)$ is the solution corresponding to the initial values $(\varphi,x_0)$ and $s_\omega\in C[-\tau,0]$ is given as usual 
    by $s_\omega(t):= 
    s(\omega+t)$. It follows that $P:X\times [0,+\infty)\to X\times [0,+\infty)$ is well defined, and a direct application of the Arzel\`a-Ascoli theorem shows that $P$ is compact. 
    In order to find a fixed point of $P$, the previous setting   suggests considering the application $f:[0,+\infty)\times X\to X$ given by 
    $$f(x_0,\varphi):= P^1(\varphi,x_0),
    $$
    where $P^1$ denotes the first coordinate of $P$. Observe, in this case, that
    Fix$(f_0)=\{ \varphi^*\}$; thus,   Corollary \ref{coro} implies the existence of  a connected subset $\C\subset \mathcal F_{[0,+\infty)}$ 
    whose projection to $[0,+\infty)$ is onto. 
    It is seen, anyway, that the weaker conclusion of  Theorem \ref{main} is already enough to deduce the existence of $\omega$-periodic solutions of the system, provided that an appropriate condition is fulfilled. 
    To this end, let us begin by noticing that, 
    if $(s,x)$ is a positive 
    $\omega$-periodic solution, then integration of the first equation yields
    $$\overline{Ds} < \overline{Ds^0} = \overline{Dv^*}
    $$
    where, for an arbitrary  continuous function we define its average $\overline w$ by $\overline w:=\frac 1\omega\int_0^\omega w(t)\, dt$. The previous inequality proves that $s(t)<v^*(t)$ for some $t$ and, as shown before, this implies  $s(t)<v^*(t)$ for all $t$. In other words, when searching for $\omega$-periodic solutions, we may restrict ourselves to look for initial values $(\varphi,x_0)\in X\times (0,+\infty)$. 
    Furthermore, since $\frac {x'}x$ is also $\omega$-periodic, it follows that 
    $$\int_0^\omega [\mu(s(t-\tau))- D(t)]\, dt =0,
    $$
    that is, using the fact that $\mu$ is strictly increasing, 
    $$ \int_0^\omega  D(t) \, dt < 
    \int_0^\omega \mu(v^*(t-\tau)) \, dt =     \int_0^\omega \mu(v^*(t)) \, dt
    $$
or, equivalently:
\begin{equation}\label{nec-suf}
    \overline D  < \overline{\mu(v^*)}.   
    \end{equation}
    As we shall see, 
    the  necessary condition (\ref{nec-suf}) is 
also sufficient for the existence of a positive $\omega$-periodic solution. 
This was already proven  in \cite{ars}, where the strategy consisted in transforming the original system into 
a one-parameter family of integro-differential equations. 
The main goal of the present section  is to emphasize the fact that the result can be retrieved as a direct consequence of the Browder Theorem in a  simple and concise way.
In order to demonstrate this assertion,  let us consider the continuous mapping 
$\Phi:[0,+\infty)\times X
\to \R$ defined
by 
$$\Phi(x_0,\varphi):= 
\overline{\mu(s)} - \overline D, 
$$
where $s$ is the first coordinate of the trajectory corresponding to the 
initial value $(\varphi,x_0)$. Condition (\ref{nec-suf}) ensures that $\Phi(0,\varphi^*)>0$; 
moreover, when $x_0>0$, the inequality $x'(t) > -D(t)x(t)$ implies, as before,
$$x(t) > e^{-\int_0^\omega D(r)\, dr}x_0=e^{-\omega\overline D}x_0\qquad t\in [0,\omega].
$$
Hence, setting $w:=v^*-s\ge 0$, it follows that 
$$w'(t) > -D(t)w(t) + cx_0\mu(s(t))
$$
where the constant $c := \frac  {e^{-\omega\overline D}}{\gamma}$ is independent of $x_0$ and $\varphi$. This means, in turn,
$$v^*(t)> w(t) > w(0)e^{-\int_0^t D(r)\, dr} + cx_0 \int_0^t e^{-\int_\xi^t D(r)\, dr} \mu(s(\xi)) \, d\xi 
$$
and, in particular
$$v^*(\omega)> 
cx_0 e^{-\omega\overline D}\int_0^\omega \mu(s(\xi))\, d\xi. 
$$
Thus, $\overline{\mu(s)}< \frac k{x_0}$ 
for some constant $k$ independent of $x_0$ and $\varphi$, which yields $\Phi(x_0,\varphi)= \overline{\mu(s)} - \overline D<0$  for any $\varphi\in X$, provided that $x_0$ is sufficiently large. 
 By  Theorem \ref{main}, for such $x_0$ there exists $\C_{0,x_0}\subset \mathcal F_{[0,+\infty)}$ connecting $\mathcal F_0=\{ (0,\varphi^*)\}$ with  $\mathcal F_{x_0}=\{x_0\} \times \textup{Fix}(f_{x_0})$. 
The continuous map $\Phi$ changes sign over $\C_{0,x_0}$, so it vanishes at some point 
$(\tilde x_0, \tilde\varphi)\in \C_{0,x_0}$ with $\tilde x_0>0$.  
Because $\tilde \varphi\in \textup{Fix}(f_{\tilde x_0})$, it follows that
$$\ln \tilde x(\omega) - \ln \tilde x_0 =
\int_0^\omega 
[\mu(\tilde s(t-\tau)) - D(t)]\, dt
$$
$$= \int_0^\omega 
[\mu(\tilde s(t)) - D(t)]\, dt
= \omega \Phi (\tilde x_0,\tilde \varphi)=0, $$
that is, $x(\omega) =\tilde x(0)$. In other words, 
$(\tilde\varphi, \tilde x_0)$
is a fixed point of the Poincar\'e map and an easy computation shows that the corresponding solution $(\tilde s,\tilde x)$ is $\omega$-periodic. 

A natural question concerning the previous proof is: 
can it  be carried out by means of Schauder's theorem only? 
Of course, this is possible using the trick described in Remark \ref{schau}, 
but this doesn't mean 
that the Schauder theorem can be directly applied 
to the Poincar\'e map. 
 In our particular situation, 
we need to observe, 
on the one hand, 
that we need to avoid the so-called \textit{trivial solution}, namely $(v^*,0)$, so the domain of $P$ should consider only values of $x_0$ larger than a positive constant.
But, on the other hand, although
the reasoning for $x_0\gg 0$ may start in a similar way,  the conclusion that $ \overline {\mu(s)}< \frac k{x_0}$  for arbitrary $\varphi\in X$ does not necessarily  imply  $x(\omega)< x_0$ for $\varphi\notin \textup{Fix}(f_{x_0})$, because 
$\int_{-\tau}^0\mu(\varphi(t))\, dt$ might be 
larger than $\omega \overline D$. 
This means that the region 
$X\times [\ee,x_0]$ may not be invariant for the Poincar\'e operator. 
A sharper argument allows to conjecture that it is possible to find an invariant set of the form 
$$\{ (\varphi,x): 0\le \varphi\le r(x)\varphi^*, \ee\le x\le x_0 
\}
$$
for some appropriate function $r\le 1$, where $r$ is initially equal to $1$ and then 
decays in such a way that $\int_{-\tau}^0 \mu(r(x)\varphi(t))\, dt < \omega \overline D$ for $x\gg 0$.  However, the choice of $r$ is not obvious and shall studied in a forthcoming paper. 

To conclude, it is worthy mentioning that the previous method does not seem to be applicable to 
different models, such as the one studied in \cite{ars2}, namely
$$\left\{ \begin{array}{l}
         s'(t)= D(t)[s^0(t)-s(t)] -\mu(s(t))\frac {x(t)}\gamma  \\
          x'(t)=e^{-d(t)}\mu(s(t-\tau))x(t-\tau)-D(t)x(t),
    \end{array}\right.
    $$
where  $d(t)=\int_{t-\tau}^{t}D(s)\,ds$ is also $\omega$-periodic. 
Indeed, here the delay appears in both 
unkwown variables; thus, although the Poincar\'e operator can be still defined,  its two coordinates 
lie on an infinite-dimensional space, so the Theorem \ref{main} does not  suffice as a tool for proving the existence of solutions.

\end{enumerate} 

\section{Conclusions} 

An infinite-dimensional version of a theorem by Browder was introduced. The focus of the paper was put on applications to  nonlinear boundary value problems.  With this in mind, new viewpoints of  known results were presented, as well as some novel results and  open problems. 

It is important to observe that, since the Browder theorem provides only a \textit{continuum} of fixed points, it is not clear how to obtain information about the topological degree of the involved maps when dealing with a system instead of a scalar equation.
This is the spirit of the open 
problem proposed in Subsection \ref{applic}.\ref{nonloc}. 
In the same direction, 
the application of Theorem \ref{main} in Subsection \ref{applic}.\ref{chemost}  sheds some light on the difficulties that may arise when dealing with other models that involve a parameter (namely, the initial condition for one of the unknown variables) in an infinite-dimensional space. 
Regarding Subsection 
\ref{applic}.\ref{pendulum},  
the problem of determining the  shape of $I(p_0)$ in the $N$-dimensional case looks like a new and challenging open question. 
The matter of nondegeneracy for the pendulum equation is 
still intriguing 
and, hopefully, the fact that 
some of the arguments in \cite{rafa-cont}  can be established in terms of Theorem \ref{main} may give a new insight to the subject.

\subsection*{Acknowledgements}
This work was supported by projects PIP 11220200100175CO, CONICET, 
UBACyT 20020190100039BA and Project TOMENADE, MATH-Amsud 21-MATH-08.
The author 
wants to express 
his gratitude to {the} anonymous reviewers for the attentive reading of the manuscript and their valuable comments.

\end{document}